\documentclass[english]{article}
\usepackage[T1]{fontenc}
\usepackage{amsmath}
\usepackage{amsthm}
\usepackage{amssymb}

\makeatletter
\theoremstyle{plain}
\newtheorem{thm}{\protect\theoremname}
\theoremstyle{plain}
\newtheorem{fact}[thm]{\protect\factname}
\theoremstyle{plain}
\newtheorem{lem}[thm]{\protect\lemmaname}
\theoremstyle{plain}
\newtheorem{rem}[thm]{Remark}

\usepackage{a4wide}

\newcommand{\C}{\mathbb {C}}
\newcommand{\R}{\mathbb R}

\newcommand{\N}{\mathbb {N}}

\def\Lou{{L^{\!\rm ou}}} 
\def\ptLou{{P_t^{\rm ou}}}%\def\ptLou{{P_t^{L^{\!\rm ou}}}}
\def\ptL{{P_t}}%\def\ptL{{P_t^L}}

\makeatother

\usepackage{babel}
\providecommand{\factname}{Fact}
\providecommand{\lemmaname}{Lemma}
\providecommand{\theoremname}{Theorem}

\begin{document}
\title{A Gaussian correlation inequality \\ for plurisubharmonic functions}
\author{F. Barthe and D. Cordero-Erausquin}
\maketitle

\begin{abstract}
    A positive correlation inequality is established for circular-invariant plurisubharmonic functions, with respect to complex Gaussian measures. The main  ingredients of the proofs are the Ornstein-Uhlenbeck semigroup, and another natural semigroup associated to the Gaussian $\overline\partial$-Laplacian.
\end{abstract}

\section{Introduction}

The motivation of the present work comes from a Gaussian moment inequality
in $\C^{n}$ due to Arias de Reyna \cite{Arias}. We will show that
his result is in fact a very particular case of a new correlation
inequality, which can be seen as the complex analogue of the following
correlation inequality for convex function in $\R^{n}$ due to Hu
\cite{Hu}: if $\mu$ is a centered Gaussian measure on $\R^n$ and if  $f,g:\R^{n}\to\R$ are \emph{convex} functions in $L^2(\mu)$ and
$f$ is \emph{even}, then
\[
\int f\,g\,d\mu\ge\int f\,d\mu\,\int g\,d\mu.
\]

We will say that a function on $\C^{n}$ is \emph{circular-symmetric} if it
is invariant under the action of $S^{1}$ (i.e. multiplication by
 complex numbers of modulus one) ; in other words  a function $f$ defined on $\C^{n}$ %\to\C\cup\{\pm\infty\}$
is circular-symmetric if
\[
f(e^{i\theta}w)=f(w)\qquad\forall\theta\in\R,\quad\forall w\in\C^{n}.
\]
A function $u:\C^n\to [-\infty,+\infty)$ is plurisubharmonic (psh) if it is upper semi-continuous and for all $a,b\in \C^n$ the function $z\in \C \mapsto u(a+zb)$ is subharmonic. Classically, a twice continuously differentiable function $u:\C^n\to \R$ is psh if and only if for all $w,z\in \C^n$
\[ \sum_{j,k} \partial^2_{z_j\overline{z_k}}u(z) \,w_j\overline{w_k}\ge 0,\]
where $\partial_{z_j}=\frac12\big(\partial_{x_j}-i\partial_{y_j} \big)$, $\partial_{\overline{z_j}}=\frac12\big(\partial_{x_j}+i\partial_{y_j} \big)$, $z=x+iy$ with $x,y\in \R^n$.
The later condition means that the complex Hessian $D^2_{\C}u$ is pointwize Hermitian semi-definite positive. 
We refer e.g. to the textbook \cite{hormander} for more details.

We consider the standard complex Gaussian measure $\gamma$ on $\C^n$, 
\[
d\gamma(w)=d\gamma_n(w)=\pi^{-n}e^{-w\cdot\overline{w}}d\ell(w)=\pi^{-n}e^{-|w|^{2}}d\ell(w),
\]
where $\ell$ denotes the Lebesgue measure on $\C^{n}\simeq\R^{2n}$
and $w\cdot w'=\sum w_{j}w'_{j}$ for $w,w'\in\C^{n}.$ For convenience,
let us introduce the following class of $L^{2}(\gamma)$ functions
with controlled growth at infinity: 
\[
\mathcal{G}:=\Big\{f:\C^{n}\to\C\,;\,f\in L_{loc}^{2}(\lambda)\,{\rm and}\,\exists\epsilon,c,C>0\,{\rm such\,\,that}\,|f(w)|\le\,e^{c|w|^{2-\epsilon}},\,\forall|w|\ge C\Big\}.
\]
 In particular any function (locally $L^{2}$) dominated by a polynomial
function on $\R^{2n}$ belongs to $\mathcal{G}$. Our main result reads as follows:
\begin{thm}[Correlation for psh functions]
\label{thm:main}Let $f,g:\C^{n}\to[-\infty,\infty)$ be two plurisubharmonic
%(psh)
 functions belonging to $\mathcal{G}$. If $f$ is %moreover
 \emph{circular-symmetric},
then
\[
\int f\,g\,d\gamma\ge\int f\,d\gamma\,\int g\,d\gamma.
\]
\end{thm}

One can extend the result by approximation to more general 
psh functions in $L^{2}(\gamma)$. The inequality also extends to arbitrary centered complex Gaussian measure, which are images of $\gamma$ by $\C$-linear maps. Indeed composing a psh function with a $\C$-linear map gives another psh function.

Let us give some direct consequences of this theorem. First, we see
that when $F$, G are holomorphic functions, or simply complex polynomial
functions, belonging to $L^{2}(\gamma)$ and $F$ is homogeneous,
then for any $\alpha,\beta\ge0$ we have 
\[
\int|F|^{\alpha}\,|G|^{\beta}\,d\gamma\ge\int|F|^{\alpha}\,d\gamma\,\int|G|^{\beta}\,d\gamma.
\]
 Indeed, if $F$ is holomorphic $f=|F|^{\beta}$ is psh, and if $F$
is homogeneous, then $f$ is also circular-symmetric. This argument can also
be used for products of the form 
\[
f:=|F_{1}|^{\alpha_{1}}\ldots|F_{k}|^{\alpha_{k}}
\]
where the $F_{j}$ are holomorphic and the $\alpha_{j}$'s are nonnegative
real numbers, so that $f$ is log-psh, in the sense that
\[
\log f(w)=\sum_{\ell=1}^{k}\alpha_{\ell}\,\log|F_{\ell}(w)|
\]
is psh. This implies that $f$ is also psh, and if the holomorphic
functions $F_{j}$ are homogeneous then $f$ is also circular-symmetric. 
\begin{thm}
Let $F_{1,}\ldots,F_{N}$ be a family of homogeneous polymomial functions on $\C^n$.
%holomorphic homogeneous functions on $\C^{n}$ belonging to $\mathcal{G}$. 
Then for any $\alpha,\ldots,\alpha_{N}\ge0$
and $k\le N-1$ we have 
\[
\int\prod_{j=1}^{N}|F_{j}|^{\alpha_{j}}\,d\gamma\ge\Big(\int\prod_{j=1}^{k}|F_{j}|^{\alpha_{j}}\,d\gamma\Big)\Big(\int\prod_{j=k+1}^{N}|F_{j}|^{\alpha_{j}}\,d\gamma\Big)\ge\prod_{j=1}^{N}\int|F_{j}|^{\alpha_{j}}\,d\gamma.
\]
\end{thm}

A standard complex Gaussian vector in $\C^n$ is a random vector taking values in $\C^n$ according to the distribution $\gamma=\gamma_n$. A random vector $X=(X_{1},\ldots,X_{N})\in \C^N$ is a centered complex Gaussian vector if there is an $n$, a standard complex Gaussian vector $G$ in $\C^n$ and a $\C$-linear map $A:\C^n\to\C^N$ such that $X=AG$. It turns out that the law for $X$ is then characterized by its complex covariance matrix $\big(\mathbb E (X_k\overline{X_\ell})\big)_{1\le k,\ell\le N}$.  
Denoting by
$a_{1,}\ldots a_{N}\in\C^{n}$ the rows the matrix of $A$ in the canonical basis,  $X_{j}=G\cdot a_{j}$. Applying the later theorem to 
the complex linear forms $F_{j}(w)=w\cdot a_{j}$ yields the following result.
\begin{thm}
Let $(X_{1},\ldots,X_{N})\in\C^{N}$ be a centered complex Gaussian
vector, and let $\alpha_{1},\ldots,\alpha_{N}\in\R^{+}$. Then, for
any $k\le N-1$

\begin{equation}
\mathbb{E}\prod_{j=1}^{N}|X_{j}|^{\alpha_{j}}\ge\Big(\mathbb{E}\prod_{j=1}^{k}|X_{j}|^{\alpha_{j}}\Big)\Big(\mathbb{E}\prod_{j=k+1}^{N}|X_{j}|^{\alpha_{j}}\Big)\label{eq:moments}
\end{equation}
 and in particular
\begin{equation}
\mathbb{E}\prod_{j=1}^{N}|X_{j}|^{\alpha_{j}}\ge\prod_{j=1}^{N}\mathbb{E}\,|X_{j}|^{\alpha_{j}}.\label{eq:genAR}
\end{equation}
In other words, among centered complex Gaussian vectors $(X_{1},\ldots,X_{N})\in\C^{N}$
with fixed diagonal covariance (i.e. $(\mathbb{E}|X_{j}|^{2})_{j\le N}$
fixed) the expectation of $\prod_{j=1}^{N}|X_{j}|^{\alpha_{j}}$ is
minimal when the variables are independent. 
\end{thm}

Inequality (\ref{eq:genAR}) is an extension of an inequality of Arias
de Reyna \cite{Arias}, who established the particular case where all
the $\alpha_{j}=2p_{j}$ are even integer by rewriting the left hand
side in terms of a permanent of a $2m$ matrix ($m=\sum p_{j}$) and
using an inequality for permanents due to Lieb. Actually, Inequality
(\ref{eq:moments}) in the case where the $\alpha_{j}$ are even
integers is equivalent to Lieb's permanent inequality, so in particular
we are giving a new proof of this inequality. 

In the next section we will introduce the tools that will be used
in the proof, that is two semi-groups: the usual Ornstein-Uhlenbeck
semi-group and another natural semi-group associated to the $\overline{\partial}$ operator
(the generator of which could be called, depending from the context, Landau or magnetic
Laplacian). In the last section we give the proof of our correlation
inequality. 

%%%%%%%%%%%%%%%%%%%%%%%%%%%%%%%%%%%%%%%%%%%%%%%%%%%%%%%%%%%%%

\section{Semi-groups}

To get the result, we will let the circular-symmetric psh function evolve along
the Ornstein-Uhlenbeck semi-group on $\C^{n}=\R^{n}+i\R^{n}\simeq\R^{2n}$
associated with the measure $\gamma$ and the scalar product $:\langle w,w'\rangle=\Re(w\cdot\overline{w'})$.
We recall that its generator is given, for smooth functions $f$, writing $w=x+iy$, $x,y\in\R^{n}$, by
\begin{eqnarray*}
\Lou f(w)&:=&\frac{1}{4}\Delta f(w)-\frac{1}{2}\langle w,\nabla\rangle f(w)\\
%&=&\frac{1}{4}\Delta f-\frac{1}{2}\sum_{j=1}^{n}(x_{j}\partial_{x_{j}}f+y_{j}\partial_{y_{j}}f),\qquad w=x+iy,\quad x,y\in\R^{n}.
&=&\sum_{j=1}^{n}\frac{1}{4}\left(\partial^2_{x_jx_j}f(w) + \partial^2_{y_jy_j}f(w)\right)-\frac{1}{2}\left(x_{j}\partial_{x_{j}}f(w)+y_{j}\partial_{y_{j}}f(w)\right).
\end{eqnarray*}
Note that the normalization differs slightly from the usual one on
$\R^{2n}$ because our Gaussian measure has complex covariance $\textrm{Id}_{n}$
but real covariance equal to $\frac12\textrm{Id}_{2n}$. Accordingly, the
spectrum of $-\Lou$ on $L^{2}(\gamma)$ is here $\frac{1}{2}\N=\{{0,\frac{1}{2},1,\ldots}$\}.
The Ornstein-Uhlenbeck semi-group $\ptLou=e^{t\Lou}$admits
the representation, for suitable functions $f:\C^{n}\to\C$,
\begin{align}
\ptLou f(z) & =\int f(e^{-t/2}z+\sqrt{1-e^{-t}}\,w)\,d\gamma(w)\label{eq:defPtLou}\\
 & =\pi^{-n}(1-e^{-t})^{-n}\int f(w)\,e^{-\frac{1}{1-e^{-t}}|w-e^{-t/2}z|^{2}}\,d\ell(w)\label{eq:defptLou2}
\end{align}

As usual in semi-group methods, it is convenient to work with a nice
stable space of functions. Here, we can for instance consider
\[
\mathcal{G}^{\infty}:=\big\{f\in C^{\infty}(\C^{n})\,;\,f\,\textrm{and all its derivatives belong to }\mathcal{\mathcal{G}\,}\big\}.
\]
Note that for $f\in\mathcal{G}$, we can \emph{define} $\ptLou f$
by (\ref{eq:defPtLou}) and then we have that $(t,z)\mapsto\ptLou f(z)$
is smooth on $(0,\infty)\times\C^{n}$, with $\ptLou f\in\mathcal{G}^{\infty}$
and $\partial_{t}\ptLou f=\Lou\ptLou f$ for every $t>0$, and $\ptLou f\to f$
in $L^{2}(\gamma)$ as $t\to0$. We refer to \cite{BGL} for details.
Let us also mention here that with formula (\ref{eq:defPtLou}) it
is readily checked that properties like convexity, subharmonicity,
pluri-subharmonicity are preserved along $\ptLou$. 

Another natural operator will be used. Indeed, since pluri-subharmonicity
is characterized through a ``$\partial_{z\overline{z}}^{2}$ operation'',
we shall also use the following differential operator on smooth functions $f$ 
on $\C^n$:
\[
Lf=\sum_{j=1}^{n}\Big(\partial_{z_{j}\overline{z_{j}}}^{2}f-\overline{z_{j}}\,\partial_{\overline{z_{j}}}f\Big)=\sum_{j=1}^{n}e^{|z|^{2}}\partial_{z_{j}}\big(e^{-|z|^{2}}\partial_{\overline{z_{j}}}f\big)
\]

Note that $Lf=0$ if (and only if, see below) $f$ is holomorphic.
Formally $L=-\overline{\partial}^{\ast}\overline{\partial}$ on $L^{2}(\gamma)$
equipped with the Hermitian structure $(f,g)=\int f\overline{g}\, d\gamma$.
More precisely, denoting for a differentiable function
\[
\partial_{\overline{z}}f=(\partial_{\overline{z_{1}}}f,\ldots,\partial_{\overline{z_{n}}}f)\quad\textrm{and}\quad\partial_{z}f:=(\partial_{z_{1}}f,\ldots,\partial_{z_{n}}f)
\]
 we have the following standard fact.
\begin{fact}[Integration by parts]
\label{fact:IPP}For regular enough functions $f,g:\C^{n}\to\C$,
for instance for functions in $\mathcal{G}^{\infty}$, we have
\begin{equation}
\int(Lf)\,\overline{g}\,d\gamma=-\int\partial_{\overline{z}}f\,\cdot\,\overline{\partial_{\overline{z}}g}\,d\gamma\label{eq:ipp1}
\end{equation}
 and in particular $\int(Lf)\,\overline{f}\,d\gamma=-\int|\partial_{\overline{z}}f|^{2}\,d\gamma\le0$.
We can also write
\begin{equation}
\int(Lf)\,g\,d\gamma=-\int\partial_{\overline{z}}f\cdot\partial_{z}g\,d\gamma=\int f(\overline{L}g)\,d\gamma\label{eq:ipp2}
\end{equation}
where $\overline{L}f:=\sum_{j=1}^{n}\Big(\partial_{z_{j}\overline{z_{j}}}^{2}f-z_{j}\,\partial_{z_{j}}f\Big)$
.
\end{fact}

Indeed, it suffices to sum over $j$ the equations
\[
\int\big[e^{|z|^{2}}\partial_{z_{j}}\big(e^{-|z|^{2}}\partial_{\overline{z_{j}}}f\big)\big]\,\overline{g}\,d\gamma=\pi^{-n}\int\partial_{z_{j}}\big(e^{-|z|^{2}}\partial_{\overline{z_{j}}}f\big)\,\overline{g}\,d\lambda(z)=-\int\partial_{\overline{z_{j}}}f\,\partial_{z_{j}}\overline{g}\,d\gamma=-\int\partial_{\overline{z_{j}}}f\,\overline{\partial_{\overline{z_{j}}}g}\,d\gamma.
\]
The assumption that $f,g\in\mathcal{G}^{\infty}$ guarantees that
the boundary terms (at infinity) in the integration by parts vanish. 

As a consequence, we see that the Gaussian measure $\gamma$ is invariant
for $L$, and actually that $L$ is a symmetric nonpositive operator
on $L^{2}(\gamma)$ with the above-mentioned Hermitian structure.
The kernel of $L$ is the  Bargmann space $\mathcal{H}_{0}$
formed by the holomorphic functions on $\C^{n}$ that belong to $L^{2}(\gamma)$. 

We want to work with the semi-group $\ptL=e^{tL}$ which is also Hermitian
(formally): 

\begin{equation}
\int(P_{t}f)\,\overline{g}\,d\gamma=\int f\overline{P_{t}g}\,d\gamma\label{eq:ptsym}
\end{equation}
%An explicit formula for the kernel of $\ptL$ can be found in \cite{AIM}.
Although we will not explicitly use it, let us
discuss a bit the (well known) spectral analysis of $L$ on the complex Hilbert
space $L^{2}(\gamma)$. This analysis is indeed fairly standard using
the ideas introduced by Landau. Following for instance the presentation
given in \cite[Section 4]{charles}, consider the ``annihilation''
operators $a_{j}=\partial_{\overline{z_{j}}}$ and their adjoints,
the ``creation'' operators $b_{j}:=a_{j}^{\ast}=\overline{z_{j}}-\partial_{z_{j}}$. 
Then  $L=-\sum_{j\le n}b_{j}\circ a_{j}$, with $[a_{j},b_{j}]=1$, and 
all these operators commute for distinct indices $j$.
Plainly, if a function $f$ and a scalar $\lambda\in \C$ satisfy $-Lf=\lambda f$, then $-L(a_jf)=(\lambda-1) a_jf$ and 
$-L(b_jf)=(\lambda+1) b_jf$.
This implies that the spectrum of $-L$ is $\N$ and that the
eigenspace associated to the eigenvalue $k\in \N$ is given by the sum of the spaces
$b^{m}\mathcal{H}_{0}$ with $m=(m_{1},\ldots,m_{n})\in \N^n$,
$|m|:=\sum_{j\le n}m_{j}=k$ and the convention $b^{m}:=b_{1}^{m_{1}}\circ\ldots\circ b_{n}^{m_{n}}$.
Moreover, if we introduce the classical projection $\Pi_{0}:L^{2}(\gamma)\to\mathcal{H}_{0}$
onto holomorphic functions

\[
\Pi_{0}f(z):=\int f(w)\,e^{z\cdot\overline{w}}d\gamma(w)=\int f(z+w)\,e^{-\overline{z}\cdot w}d\gamma(w),
\]
then the projector $\Pi_{k}$ on the $k$-eigenspace can be expressed in terms of $\Pi_0$ and the creation and annihilation operators. This allows to compute the reproducing kernel of $\Pi_k$, in terms of classical families of orthogonal polynomials.
Next, one can sum over $k$ and obtain the kernel $K_{t}(z,w)$ for
$e^{tL}=\sum_{k}e^{-kt}\Pi_{k}$. Only the formula for $K_{t}$ will
be useful in the sequel and we shall actually \emph{check} below that
this suggested formula is indeed the kernel of $e^{tL}$. 

 An explicit formula for $K_{t}$ can be found in \cite{AIM}: setting

\begin{equation}
K_{t}(z,w):=\frac{1}{\pi^{n}(1-e^{-t})^{n}}\exp\Big(z\cdot\overline{w}-\frac{e^{-t}|z-w|^{2}}{1-e^{-t}}-|w|^{2}\Big),\label{eq:ptKernel1}
\end{equation}
then
\begin{align}
\ptL f(z) & =\int f(w)\,K_{t}(z,w)d\ell(w)\label{eq:ptKernel2}\\
 & =(1-e^{-t})^{-n}\int f(w)\,e^{z\cdot\overline{w}-\frac{e^{-t}|z-w|^{2}}{1-e^{-t}}}\,d\gamma(w)\nonumber 
\end{align}
Next, let us note that by performing the change of variable $w=z+\sqrt{1-e^{-t}\,}\xi$
for fixed $z$ we find \texttt{
\begin{equation}
\ptL f(z)=\int f(z+\sqrt{1-e^{-t}\,}\xi)\,e^{-\sqrt{1-e^{-t}}\,\overline{z}\cdot\xi}\,d\gamma(\xi).\label{eq:Pt}
\end{equation}
}On this formula, we see immediately that $P_{0}=\textrm{Id and \ensuremath{P_{\infty}=\Pi}}_{0}$. 

To avoid discussions regarding unbounded operators and existence of
semi-groups, we will proceed in the opposite direction and use the previous
formula to \emph{define} $\ptL f$. Actually, to be fair, we should
mention that later, in the proof of our result, we only need to work
with smooth functions $f\in\mathcal{G^{\infty}}$; these functions
provide nice initial data, ensuring existence and uniqueness of strong
solutions for the semi-group equation. Nevertheless, we feel it is
of independent interest to start from the integral formula (\ref{eq:ptKernel2})
or (\ref{eq:Pt}) and derive from it the semi-group properties ; we will first check that~\eqref{eq:ptKernel2} solves~\eqref{eq:semigroup} indeed. The
drawback is that some properties that are obvious (formally) for $e^{tL}f$
need to be checked thoroughly when using this kernel representation, in particular because the kernel~\eqref{eq:ptKernel1} is not Markovian. 

Formulas~\eqref{eq:ptKernel2}-\eqref{eq:Pt} make sense pointwise, for $z\in\C^n$ fixed, as long as $f\in L^2(\gamma)$; actually we have the pointwise estimate $|\ptL f(z)|\le C_t(z) \|f\|_{L^2(\gamma)}$ for some constant $C_t(z)>0$ depending on $t$ and $z$. In order to derive some properties of $P_t f$, a stronger integrability condition (as $f\in\mathcal G$) will be assumed. 

\begin{lem}
\label{lem:defPtL}Given $f\in\mathcal{G}$, let us  \emph{define}
$\ptL f$ using the formula (\ref{eq:ptKernel2}) or (\ref{eq:Pt}).
Then one has that $(t,z)\to\ptL f(z)$ is smooth on $(0,\infty)\times\C^{n}$,
with for any $t>0$, $\ptL f\in\mathcal{G_{\infty}}$ and 
\begin{equation}
    \label{eq:semigroup}
\partial_{t}\ptL f=L\ptL f.
\end{equation}

 Moreover we have
\[
\|\ptL f\|_{L^{2}(\gamma)}\le\|f\|_{L^{2}(\gamma)}
\]
 and $P_{t}f\to f$ in $L^{2}(\gamma)$ as $t\to0$. 
\end{lem}

\begin{proof}
It is readily checked that, for any fixed $w\in\C^{n}$ one has
\[
\partial_{t}K_{t}(\cdot,w)=L\,K_{t}(\cdot,w).
\]
Moreover, for $T,R,k>0$ fixed, there exists constants $c=c(T,R,k),C=C(T,R,k)$
such that for $F=\partial K_{t}$, or $F=K_t$ or else $F$ being any
of the first $k$th partial derivatives of $K_{t}$, it holds that
$|F(z,w)|\le Ce^{-c|w|^{2}}$ for all $w\in\C^n$, all $t\in(0,T)$ and all $|z|\le R.$
From this and the definition of $\mathcal{G}$, we can call upon dominated
convergence to conclude to the smoothness of $(t,z)\to\ptL f(z)$ and
to the fact that $\ptL f\in\mathcal{G}^{\infty}$ with $\partial_{t}\ptL f=L\ptL f.$

Regarding the contraction property, we want to avoid direct computations or spectral arguments, 
%(using Fourier)
 and so we  make a detour and use some obvious
but important properties of $\ptL$. 

Fist, we will use the semi-group property: for $t,s>0$ that $\ptL\circ P_{s}=P_{t+s}$.
This can be seen in two ways. One can invoke that for a \emph{smooth}
function $F:=P_{t_{0}}f\in\mathcal{G}^{\infty}$ then there is existence
and uniqueness for the equation $\partial_{t}F_{t}=LF_{t}$ with initial
condition $F_{0}=F$ and we have seen that $(P_{t_{0}+t}f)$ is a
solution; from there one can conclude. Or else, in a more pedestrian
way, one can check that
\[
\int K_{t}(z,w)K_{s}(\xi,z)\,d\ell(z)=K_{t+s}(\xi,w).
\]
For this, one may use that given $z,\xi\in\C^{n}$ and $c\in\C$ with $\Re(c)>0$,
\begin{equation}\label{laplace}
\int e^{z\cdot w+\xi\cdot\overline{w}}e^{-c|w|^{2}}\,d\ell(w)=\pi^{n}c^{-n}\,e^{z\cdot\xi\,/c}.
\end{equation}

Next, recall that $\ptL$ is Hermitian, in the sense (\ref{eq:ptsym}),
on $\mathcal{G}\subset L^{2}(\gamma)$; this can be seen directly
from the integral formula since $K_{t}(z,w)=\overline{K_{t}(w,z)}$.
Finally, we will use that $\|\ptL f\|_{L^{2}(\gamma)}$ decreases
for $t\in(0,\infty)$. This is immediate from the non-positivity of $L$ since for
$t>0$ we have $\ptL f\in\mathcal{G}^{\infty}$ and
\[
\frac{d}{dt}\int|P_{t}f|^{2}d\gamma=2 \int(L\ptL f)\overline{\ptL f}\,d\gamma=-2\int|\partial_{\overline{z}}\ptL f|^{2}\,d\gamma\le0.
\]
 So for our $f\in\mathcal{G}$ and $t>0$ we have 
\[
\|\ptL f\|_{L^{2}(\gamma)}^{2}=\int(P_{2t}f)\,\overline{f}\,d\gamma\le\|P_{2t}f\|_{L^{2}(\gamma)}\,\|f\|_{L^{2}(\gamma)}\le\|\ptL f\|_{L^{2}(\gamma)}\,\|f\|_{L^{2}(\gamma)}
\]
 which implies the contraction property in $L^{2}(\gamma)$.

To prove the continuity at $t=0$ in $L^2(\gamma)$ we first assume that $f$ is smooth
and compactly supported. Using (\ref{eq:Pt}) we see that $\ptL f$
converge point-wise to $f$ and that for $t\in(0,1)$ we have $|\ptL f(z)|\le C\,e^{c|z|}$
for some constant $c,C>0$; so we can conclude by dominated convergence.
For $f\in\mathcal{G}$ and $\epsilon>0$, introduce $g$ smooth compactly
supported such that $\|f-g\|_{L^{2}(\gamma)}\le\epsilon$ and let
$\delta>0$ be such that $t\le \delta$ ensures that $\|\ptL g-g\|_{L^{2}(\gamma)}\le\epsilon$ holds.  For $t\le \delta$,
\[
\|\ptL f-f\|_{L^{2}(\gamma)}\le\|\ptL f-\ptL g\|_{L^{2}(\gamma)}+\|\ptL g-g\|_{L^{2}(\gamma)}+ \|g-f\|_{L^{2}(\gamma)}\le2\|f-g\|_{L^{2}(\gamma)}+\epsilon\le3\epsilon.
\]
This establishes the desired continuity. 
\end{proof}

\begin{rem}[Contraction property]
$~$
\begin{enumerate}
    \item We observe that some results, which can be deduced from the spectral decomposition and Hilbertian analysis, may be derived in a soft way thanks to flexible semigroup techniques. 
    We have proved that starting from formula~(\ref{eq:ptKernel2}) we have $\|\ptL f\|_{L^2(\gamma)}\le \|f\|_{L^2(\gamma)}$ on the dense subspace $\mathcal G$, which together with the pointwise estimate given just before the previous Lemma implies by density that
    $$\|\ptL\|_{L^2(\gamma)\to L^2(\gamma)}\le 1.$$
    Formally, by letting $t\to \infty$ in $\|\ptL\|_{L^2(\gamma)\to L^2(\gamma)}\le 1$ we recover that
$$\|\Pi_0\|_{L^2(\gamma)\to L^2(\gamma)}\le 1.$$ Actually, the convergence of $P_t$ towards $\Pi_0$ can be quantified rigorously through H\"ormander's inequality,
$$\|\varphi - \Pi_0\varphi\|_{L^2(\gamma)}^2\le \|\partial_{\overline z} \varphi\|_{L^2(\gamma)}^2=\int (-L\varphi)\varphi\, d\gamma$$
valid for any suitable $\varphi$, for instance for $\varphi\in\mathcal G^\infty$. Note that from formula~\eqref{eq:Pt}, $P_t$ acts as the identity on holomorphic functions, so $\ptL \Pi_0=\Pi_0 \ptL=\Pi_0$. A classical Gr\"onwall type argument (using the previous Lemma to justify the computation of $\frac{d}{dt}\int |P_t (f-\Pi_0 f)|^2\, d\gamma$) ensures that for $f\in\mathcal G$ and $t\ge 0$ 
$$\|\ptL f - \Pi_0 f\|_{L^2(\gamma)}^2\le e^{-2t}\, \|f-\Pi_0 f\|_{L^2(\gamma)}^2
.$$
\item In analogy with the Markovian case $\ptLou$ we may wonder if $P_t$ is also a contraction on some $L^p(\gamma)$.
However, for any $p\neq 2$ we have
    $$\|P_t \|_{L^p(\gamma)\to L^p(\gamma)}=+\infty,$$ 
    as it can be seen by taking, in dimension $n=1$, for $a\in\R$,
    $$f_a(w):=e^{aw+\overline w}, \qquad w\in \C.$$
    Indeed,repeated applications of~\eqref{laplace} with $c=1$  show, setting $s_t:=\sqrt{1-e^-t}$ and using~\eqref{eq:Pt}, that $P_t f_a(z)=e^{s_t^2 a} e^{az+(1-s_t^2)\overline z}$ and  that
    $$\frac{\|P_t f_a\|_{L^p(\gamma)}^p}{\|f_a\|_{L^p(\gamma)}^p}=C(t,p)\,  e^{a\, s_t^2 (p-p^2/2)} .$$
\end{enumerate}

\end{rem}

The next result describes how derivatives and $P_t$ commute, an important issue in semi-group methods. 
\begin{lem}[Commutation relations]
\label{lem:commutation} For any suitable $f$, say $f\in\mathcal{G^{\infty}}$,
and $t>0$ we have for every $1\le j\le n$ and $z\in\C^{n},$
\[
\partial_{z_{j}}(\ptL f(z))=\ptL(\partial_{z_{j}}f)(z)\,\qquad\textrm{and}\qquad\partial_{\overline{z_{j}}}(\ptL f(z))=e^{-t}\ptL(\partial_{\overline{z_{j}}}f)(z).
\]
\end{lem}

\begin{proof}
We use (\ref{eq:Pt}). The first equality is obvious. For the second one, setting $s_t=\sqrt{1-e^{-t}}$,
we have
\[
\partial_{\overline{z_{j}}}(P_{t}f)(z)=P_{t}(\partial_{\overline{z_{j}}}f)(z)-s_{t}\int f(z+s_{t}\xi)\,e^{-s_{t}\,\overline{z}\xi}\,\xi_{j}\,d\gamma(\xi),
\]
and 
\begin{align*}
\pi^{-n}\int f(z+s_{t}\xi)\,e^{-s_{t}\,\overline{z}\xi}\,\xi_{j}e^{-\xi\cdot\overline{\xi}}\, d\ell(\xi)&=-\pi^{-n}\int f(z+s_{t}\xi)\,e^{-s_{t}\,\overline{z}\xi}\:\partial_{\overline{\xi_{j}}}(e^{-\xi\cdot\overline{\xi}})\,d\ell(\xi)\\
&=s_{t}\int(\partial_{\overline{z_{j}}}f)(z+s_{t}\xi)\,e^{-s_{t}\,\overline{z}\xi}\, d\gamma(\xi)=s_{t}\,P_{t}(\partial_{\overline{z_{j}}}f)(z).
\end{align*}
\end{proof}
Now, and for the rest of this section, we focus on the case of circular-symmetric
functions. Given $\theta\in\R$ and a function $f$ we denote $f_{\theta}$
the function $f_{\theta}(w)=f(e^{i\theta}w)$. Note that
\begin{equation}
P_{t}(f_{\theta})=(P_{t}f)_{\theta}.\label{eq:invariance}
\end{equation}
Recall that a function $f$ is said to be circular-symmetric if $f_{\theta}=f$
for every $\theta$. It is worth noting that a holomorphic function
on $\C^{n}$ is necessarily constant when circular-symmetric. Indeed if $h:\C \to \C$ has both properties then invariance and the Cauchy formula give $h(1)=\int_{0}^{2\pi}h(e^{i\theta})d\theta/(2\pi)=h(0)$; next if $f:\C^n \to \C$ is holomorphic and circular symmetric, then for any $(z_1,\ldots,z_n)\in  \C^n$, the function $h:z\in \C \mapsto f(zz_1,\ldots,zz_n)$ is also holomorphic and circular-symmetric, hence $f(z_1,\ldots,z_n)=h(1)=h(0)=f(0)$.
Accordingly
if $f\in L^{2}(\gamma)$ is circular-symmetric then $\Pi_{0}f\equiv\int f\,d\gamma$
since the gaussian density is also circular-symmetric. Actually, much more
can be said, as we shall see below. 

Let us first investigate the relation between $L$ and $\Lou.$ Note
that one can write

\begin{equation}
Lf=\Lou f+\frac{i}{2} \sum_{j=1}^n \big(y_{j}\partial_{x_{j}}f-x_{j}\partial_{y_{j}}f\big).\label{eq:L-Lou}
\end{equation}
So we have $\Lou=\Re(L)=\frac{L+\text{\ensuremath{\overline{L}}}}{2}$
where $\overline{L}f=\sum_{j=1}^{n}\Big(\partial_{z_{j}\overline{z_{j}}}^{2}f-z_{j}\,\partial_{z_{j}}f\Big)$
has a kernel formed by the anti-holomorphic functions. The operators
$L$ and $\overline{L}$ are Hermitian symmetric, whereas $\Lou$
is symmetric for the real and the Hermitian product, and preserves the subspace of real valued functions. %Actually, unlike $L$, the operator $\Lou$ isMarkovian, a crucial property for semi-group methods.
 As we said,
its spectrum is  $-\frac{1}{2}\N$, as can be seen also from the
formula $\Lou=\frac{L+\text{\ensuremath{\overline{L}}}}{2}$.
 Let us illustrate this on two examples, obtained by applying the creation operator $b_1=\overline{z_1}-\partial_{z_1}$ to the holomorphic functions $z\mapsto 1$, and $z\mapsto z_1$. The function $z\mapsto \overline{z_1}$ is an eigenfunction for $L$ with eigenvalue $-1$, for $\overline{L}$ with eigenvalue $0$, and for $\Lou$ with eigenvalue $-1/2$. The function $z\mapsto |z_1|^2-1$ is an eigenfunction for $L$ with eigenvalue $-1$, for $\overline{L}$ with eigenvalue $-1$, and for $\Lou$ with eigenvalue $-1$.

The special role played by circular-symmetric functions is due to the fact that these operators, and the associated semi-groups, coincide for them. 
\begin{lem}[Action of $L$ and $\ptL$ on circular-symmetric functions]
\label{lem:actionONsymm} If $f$ is a smooth circular-symmetric function,
then we have 
\[
Lf=\overline{L}f=\Lou f.
\]
In particular we have, when $f,g\in\mathcal{G^{\infty}}$ and $f$
is circular-symmetric, 

\[
\int(Lf)g\,d\gamma=\int fLg\,d\gamma=-\int\partial_{z}f\cdot\partial_{\overline{z}}g\,d\gamma.
\]
 Also, if $f\in L^2(\gamma)$ is circular-symmetric then we have 
\[
P_{t} f=\ptLou f
\]
 for every $t\ge0$. 
\end{lem}

\begin{proof}
Writing $w=x+iy$, $x,y\in\R^{n}$, the symmetry rewrites as $f((\cos(\theta)x-\sin(\theta)y)+i(\cos(\theta)y+i\sin(\theta)x))=f(x+iy)$.
Taking the derivative at $\theta=0$ we find
\[
\sum_{j=1}^{n}\Big(-y_{j}\partial_{x_{j}}f(x+iy)+x_{j}\partial_{y_{j}}f(x+iy)\Big)=0,
\]
 and this for every $x,y\in\R^{n}$. This implies in view of (\ref{eq:L-Lou})
that $Lf=\Lou f=\overline{L}f$. Next, for any smooth function $g$
we have, using that $Lf=\overline{L}f$ and (\ref{eq:ipp2}),
\[
\int(Lf)g\,d\gamma=\int(\overline{L}f)g\,d\gamma=\int fLg\,d\gamma=-\int\partial_{z}f\cdot\partial_{\overline{z}}g\,d\gamma.
\]
Although it is formally trivial that equality of $L$ and $\Lou$
on circular-symmetric functions implies equality of the semi-groups $\ptL$
and $\ptLou$, a bit more should be said since we defined the semi-group
using the explicit formula~\eqref{eq:ptKernel2}. And it is anyway instructive to compute
the kernels on circular-symmetric functions. Denote by $K_{t}^{{\rm ou}}$
the kernel of the Ornstein-Uhlenbeck semi-group that we recalled above:
$K_{t}^{{\rm ou}}(z,w)=\pi^{-n}(1-e^{-t})^{-n}\,e^{-\frac{1}{1-e^{-t}}|w-e^{-t/2}z|^{2}}\,$.
So we have, setting $c_{t}:=e^{-t/2}$ and $s_{t}:=\sqrt{1-e^{-t}}$,
\[
K_{t}^{{\rm ou}}(z,w)=\pi^{-n}s_{t}^{-2n}\,e^{-s_{t}^{-2}|w|^{2}-s_{t}^{-2}c_{t}^{2}|z|^{2}}\,e^{s_{t}^{-2}c_{t}(w\cdot\overline{z}+\overline{w}\cdot z)}
\]
and 
\[
K_{t}(z,w)=\pi^{-n}s_{t}^{-2n}e^{-s_{t}^{-2}|w|^{2}-s_{t}^{-2}c_{t}^{2}|z|^{2}}\,e^{s_{t}^{-2}(c_{t}^{2}w\cdot \overline{z}+\overline{w}\cdot z)}.
\]
 Note that only the last exponentials differ in these two formulas.
When $f$ is circular-symmetric, in order to check that $\ptL f=\ptLou f$
it suffices to check that for fixed $w,z,t$ one has

\[
\frac{1}{2\pi}\int_{0}^{2\pi}K_{t}(z,e^{i\theta}w)\:d\theta=\frac{1}{2\pi}\int_{0}^{2\pi}K_{t}^{{\rm ou}}(z,e^{i\theta}w)\:d\theta.
\]
Observe that for $a,b\in\C$, we have
\[
\frac{1}{2\pi}\int_{0}^{2\pi}e^{ae^{i\theta}+be^{-i\theta}}\,d\theta=\frac{1}{2\pi}\int_{0}^{2\pi}\sum_{p,q\in\mathbb{N}}\frac{a^{p}b^{q}}{p!q!}e^{i(p-q)\theta}d\theta=\sum_{n\ge0}\frac{(ab)^{n}}{(n!)^{2}}=B(ab)
\]
with $B(x):=\sum_{n\ge0}\frac{x^{n}}{(n!)^{2}}$. Therefore, we have
\begin{align*}
\frac{1}{2\pi}\int_{0}^{2\pi}K_{t}(z,e^{i\theta}w)\:d\theta & =\pi^{-n}s_{t}^{-2n}\,e^{-s_{t}^{-2}|w|^{2}-s_{t}^{-2}c_{t}^{2}|z|^{2}}B(s_{t}^{-4}c_{t}^{2}|w\cdot\overline{z}|^{2})\\
 & =\frac{1}{2\pi}\int_{0}^{2\pi}K_{t}^{{\rm ou}}(z,e^{i\theta}w)\:d\theta,
\end{align*}
as wanted. 
\end{proof}
%

%%%%%%%%%%%%%%%%%%%%%%%%%%%%%%%%%%%%%%%%%%%%%%%%%%%

\section{Proof of Theorem \ref{thm:main}}

First, let us note that we can assume that $g$ is smooth, and actually
that $g\in\mathcal{G}^{\infty}$. Indeed, if $g\in\mathcal{G}$ then
$\ptLou g\in \mathcal{G}^{\infty}$ and we mentionned that $\ptLou g$
converges to $g$ in $L^{2}(\gamma)$ and therefore also in $L^{1}(\gamma),$
as $t\to0$. Consequently, if we know the conclusion for a function
in $\mathcal{G}^{\infty}$, then
\[
\int f\ptLou g\,d\gamma\ge\int f\,d\gamma\,\int\ptLou g\,d\gamma
\]
and by passing to the limit when $t\to0$ we know it also for $g\in\mathcal{G}.$
For the same reason, we can assume that $f\in\mathcal{G^{\infty}},$
recalling that $\ptLou f$ is circular-symmetric when $f$ is. 

So in the sequel, we are given two psh functions $f,g\in\mathcal{G^{\infty}},$
with $f$ circular-symmetric. 

As in the proof of the correlation for convex functions \cite{Hu},
we will compute some kind of second derivative in $t$ for integrals
involving $\ptLou f$; recall that $\partial_{t}\ptLou f=\Lou\ptLou f$.
The novelty is that, along the way, we will also use $P_{t} f$
which satisfies $\partial_{t}P_{t}f=L\ptL f$ (Lemma \ref{lem:defPtL}). 

Consider
\[
\alpha(t):=\int(\ptLou f)\,g\,d\gamma=\int(\ptL f)\,g\,d\gamma\,\in\R.
\]
The function $\alpha$ is, by construction, smooth on $(0,\infty)$
and continuous on $[0,\infty)$ (see Lemma \ref{lem:defPtL} for the
continuity at zero). 
%; actually our assumptions ensure that $\alpha$ is also smooth at zero, but we don't need it.
 Since $\ptLou f$ tends
to the constant $\int f\,d\gamma$ when $t\to\infty$, we have 
\[
\alpha(t)\to\int f\,d\gamma\,\int g\,d\gamma
\]
 In order to conclude, it suffices to show that $\alpha$ decreases. Actually we will prove that $\alpha$ is convex; which is enough, since a convex function with a bounded limit at $+\infty$ cannot increase. It holds
\begin{equation}\label{eq:alphaprime}
\alpha'(t)=\int\Lou(\ptLou f)\,g\,d\gamma=\int L\ptL f\,g\,d\gamma.
\end{equation}
Since $\ptL f$ is also circular-symmetric, we can invoke Lemma~\ref{lem:actionONsymm}
and write
\[ 
\alpha'(t)=-\int\partial_{z}\ptL f\cdot\partial_{\overline{z}}g\,d\gamma.
\] 
 Next, using the first commutation relation from Lemma~\ref{lem:commutation}
we get

\[
\alpha'(t)=-\sum_{j=1}^{n}\int\ptL(\partial_{z_{j}}f)\partial_{\overline{z}_{j}}g\,d\gamma
\]
We stress that $\partial_{z}f$ is no longer circular-symmetric, so we cannot
exchange $\ptL$ and $\ptLou.$ The second derivative of $\alpha$ is, using Fact \ref{fact:IPP},
\[
\alpha''(t)=-\sum_{j=1}^{n}\int L(\ptL(\partial_{z_{j}}f))\partial_{\overline{z}_{j}}g\,d\gamma=\sum_{j=1}^{n}\int\partial_{\overline{z}}(\ptL(\partial_{z_{j}}f))\cdot\partial_{z}(\partial_{\overline{z}_{j}}g)\,d\gamma=\sum_{j,k=1}^{n}\int\partial_{\overline{z_{k}}}\ptL(\partial_{z_{j}}f)\,\partial_{z_{k}\overline{z_{j}}}^{2}g\,d\gamma.
\]
Using the commutation relation from Lemma~\ref{lem:commutation}
we can write

\[
\alpha''(t)=\sum_{j,k=1}^{n}\int\partial_{\overline{z_{k}}z_{j}}^{2}(\ptL f)\,\partial_{z_{k}\overline{z_{j}}}^{2}g\,d\gamma=\int\textrm{Tr}\Big((D_{\C}^{2}\ptL f)(z)\,\,(D_{\C}^{2}g)(z)\Big)\,d\gamma
\]
where for a $C^{2}$ function $h$ on $\C^{n}$ the notation $D_{\C}^{2}h(z)$
refers to the Hermitian matrix $\Big(\partial_{z_{j}\overline{z_{k}}}^{2}h(z)\Big)_{j,k\le n}$.
Since $\ptL f={\ptLou}f$ and $g$ are psh, the corresponding
matrices are nonnegative Hermitian matrices, which means that the
trace of their product is still nonnegative. This shows that $\alpha''\ge0$
and finishes the proof of the theorem. 

\qed

\medskip
We would like to conclude with a discussion of the differences
between the real case and the complex case. After all, we are computing
second derivatives of the same object 
\[
\alpha(t)=\int({\ptLou}f)\,g\,d\gamma
\]
along the Ornstein-Uhlenbeck semi-group exactly as in the case of
convex functions, so what is going on? 

In both cases we prove that $\alpha$ decreases by showing that $\alpha'\le0$
using the next derivative somehow, but we compute these derivatives
differently. The argument for convex function goes as follows. A direct
computation and usual commutation properties show that, if  $\int \nabla f\, d\gamma=0$, which is the case when $f$ is even, then
\[
\alpha'(t)=-e^{-t/2}\int_{t}^{\infty}\Big(\int\textrm{Tr}\Big((D^{2}P_{s}^{{\rm ou}}f)(z)\,\,(D^{2}g)(z)\Big)\,d\gamma(z)\Big)e^{s/2}ds
\]
where $D^{2}$ refers to the usual (real) Hessian on $\R^{2n}$; from
this we conclude to the correlation for convex functions. On the other
hand, we have proved, when $f$ is circular-symmetric, that 
\[
\alpha'(t)=-\int_{t}^{\infty}\Big(\int\textrm{Tr}\Big((D_{\C}^{2}P_{s}^{{\rm ou}}f)(z)\,\,(D_{\C}^{2}g)(z)\Big)\,d\gamma(z)\Big)\,ds.
\]
Note that we have used here that $\alpha'(t)$ tends to 0 when $t\to +\infty$: this follows from the fact that $\alpha$ is convex and has a finite limit at $+\infty$, and can also be seen from  \eqref{eq:alphaprime} since $\ptLou f$ tends to a constant when $t\to +\infty$.
It is because we wanted to work with complex derivatives that we aimed
at inserting $L$ in place of $\Lou$; recall that $\partial_{z_{j}}f$
need not be circular-symmetric even when $f$ is, although the second derivatives
are again circular-symmetric.

Finally, let us observe that if we consider in dimension $1$ the
circular-symmetric psh functions $f(w)=|w|^{1/3}$ and $g(w)=|w|^{4}$ on $\C\simeq\R^{2}$,
then, as we already mentioned, 

\[
\textrm{Tr}\Big((D_{\C}^{2}f)\,\,(D_{\C}^{2}g)\Big)=\Delta f\,\Delta g\ge0,
\]
but a direct computation shows that
\[
{\rm Tr}\Big((D^{2}f)(z)\,\,(D^{2}g)(z)\Big)=-\frac{4}{3}|z|^{1/3}\le0\,\quad\forall z\in\C.
\]
Of course, this discrepancy cannot hold at all times for $\ptLou f$
in place of $f$ (and moreover $f$ is not smooth at zero, although
this is not really an issue). But it suggests that the two formulas
above for $\alpha'(t)$ are indeed quite different. 

%%%%%%%%%%%%%%%%%%%%%%%%%%%%%%%%%%%%%%%%%%%%%%%%%%%%%%%%%

\bigskip 

\noindent {\bf Acknowledgements:} We thank Laurent Charles, Vincent Guedj and Pascal Thomas for useful discussions, and for pointing out to us relevent references.

\medskip 

 {\sl Franck Barthe}: Institut de Mathématiques de Toulouse (UMR 5219). Université de Toulouse \& CNRS. UPS, F-31062 Toulouse Cedex 9, France.

{\sl Dario Cordero-Erausquin}: Institut de Mathématiques de Jussieu (UMR 7586), Sorbonne Université \& CNRS, 75252 Paris Cedex 5, France. 

\end{document}